\documentclass[12pt]{article}
\usepackage[margin=.75in]{geometry}
\usepackage{mathrsfs}
\usepackage{latexsym}
\usepackage{amsmath}
\usepackage{amssymb}
\usepackage{graphicx}
\usepackage{wrapfig}
\usepackage{fancybox}
\usepackage{bm}
\usepackage{amsfonts}
\usepackage{amsmath}
\usepackage{amsthm}
\usepackage{tikz-cd}
\newtheorem{theorem}{Theorem}[section]
\newtheorem{corollary}{Corollary}[section]
\newtheorem{lemma}[theorem]{Lemma}
\theoremstyle{definition}
\newtheorem{definition}{Definition}[section]
\theoremstyle{definition}
\newtheorem{proposition}[theorem]{Proposition}
\newtheorem{exmp}{Example}[section]

\title{Counting rotational subsets of the circle $\mathbb{R}/ \mathbb{Z}$ under the angle-multiplying map $t\mapsto dt$}

\author{Yee Ern Tan}
\begin{document}
\date{June 24, 2023}
\maketitle
\abstract{
A rotational set is a finite subset $A$ of the unit circle $\mathbb{T}=\mathbb{R}/ \mathbb{Z}$ such that the angle-multiplying map $\sigma_{d}:t\mapsto dt$ maps $A$ onto itself by a cyclic permutation of its elements. Each rotational set has a geometric rotation number $p/q$. These sets were introduced by Lisa Goldberg to study the dynamics of complex polynomial maps. In this paper we provide a necessary and sufficient condition for a set to be $\sigma_{d}$-rotational with rotation number $p/q$. As applications of our condition, we recover two classical results and enumerate $\sigma_d$-rotational sets with rotation number $p/q$ that consist of a given number of orbits.
}

\begin{section}{Introduction and preliminaries}

Let $\mathbb{T}=\mathbb{R}/\mathbb{Z}$ be the unit circle parametrized by angle measured in revolutions. Impose a total order on $\mathbb{T}$ by choosing representatives in $[0,1)\subset\mathbb{R}$. That is, given $s,t\in\mathbb{T}$ we say that $s<t$ if and only if $s'<t'$ where $s',t'\in [0,1)$ correspond to the equivalence classes $s,t$ respectively. Likewise we impose a total order on $\mathbb{Z}/q\mathbb{Z}$ by choosing representatives in $\{0,1,\dots,q-1$\}. Fix an integer $d\geq 2$ and define the angle-multiplying map $\sigma_{d}$ : $\mathbb{T}\rightarrow\mathbb{T}$ with $\sigma_{d}(t)=dt, \forall t\in\mathbb{T}$. The {\it orbit} of an element $ t\in\mathbb{T}$ is the set $\{\sigma_{d}^{k}(t):k=0,1,2,\dots\}$. 

\begin{definition} Let $A=\{t_{0}<t_{1}<t_{2}<\dots<t_{q-1}\}$ be a finite subset of $\mathbb{T}$ indexed in increasing order, where indices are viewed as elements of $\mathbb{Z}/ q\mathbb{Z}$. Then $A$ is $\sigma_{d}$-{\it rotational} if for some fixed $0\neq p\in \mathbb{Z}/ q\mathbb{Z}$ we have $\sigma_{d}(t_{i})=t_{i+p}$ for every $t_{i}\in A.$ Each rotation set is assigned a \textit{rotation number}, which is the rational number $p/q$ whose expression may be further reducible. A set is \textit{rotational} if it is $\sigma_d$-rotational for some $d\geq 2$.
\end{definition}

Rotational sets were introduced by Goldberg in [2] to study and classify polynomial maps $\mathbb{C}\to\mathbb{C}$ in the setting of complex dynamics [3]. Goldberg shows in [2] that rotational sets are uniquely determined by their rotation number and {\it deployment sequence}, which describes the placement of the set with respect to the fixed points of $\mathbb{T}$ under $\sigma_{d}$. This allows for an enumeration of $\sigma_{d}$-rotational orbits with rotation number $p/q$. (Rotational sets are rotational orbits or the union of rotational orbits.) This is generalized in [6] by Petersen and Zakeri, who enumerate periodic orbits under $\sigma_{d}$ with any {\it combinatorial type}. That is, periodic orbits which are generally permuted under $\sigma_{d}$ are considered. We continue progress on the combinatorics of rotational sets by introducing a new characterization of rotational sets (Theorem \ref{thm}). Two results from [2] promptly follow:

\begin{enumerate}
\item There are ${{d-2+q}\choose d-2}$ $\sigma_{d}$-rotational orbits with rotation number $p/q$ (in lowest terms).
\item A $\sigma_{d}$-rotational set contains at most $d-1$ orbits.
\end{enumerate}

In Theorem \ref{main} we enumerate all $\sigma_{d}$-rotational sets with rotation number $p/q$. The proof is constructive in the sense that it suggests an algorithm for generating the rotational sets.

\begin{exmp} The set $\{\frac{1}{15}, \frac{2}{15}, \frac{4}{15}, \frac{8}{15}\}$ is $\sigma_{2}$-rotational with rotation number $\frac{1}{4}$. The set
$$
A=\big\{ \frac{8}{26}, \frac{17}{26}, \frac{20}{26}, \frac{23}{26}, \frac{24}{26}, \frac{25}{26} \big\}
$$
is $\sigma_{3}$-rotational with rotation number $ \frac{4}{6}=\frac{2}{3}$. Notice that $A$ is the union of two $\sigma_3$-rotational orbits $ \{\frac{8}{26},\frac{24}{26},\frac{20}{26}\}$ and $ \{\frac{17}{26},\frac{25}{26},\frac{23}{26}\}.$ Both smaller orbits also have rotation number $\frac{2}{3}$.
\end{exmp}

A subset of $\mathbb{Z}/q\mathbb{Z}$ consists of \textit{consecutive} elements if it has the form $\{i,i+1,i+2,\dots\}$ for some $i\in \mathbb{Z}/q\mathbb{Z}$. A collection indexed by elements in $\mathbb{Z}/q\mathbb{Z}$ consists of \textit{consecutive} elements if the indices are consecutive.

\begin{proposition}\label{inter} {\it Let} $A=\{t_{0}<t_{1}<t_{2}<\dots<t_{q-1}\}$ {\it be a} $\sigma_{d}$-{\it rotational set with rotation number $p/q$, that is, $\sigma_d(t_i)=t_{i+p}$. Then} $A$ {\it is the union of} $n=\text{gcd}(p,q)$ {\it orbits. Each orbit in the union is} $\sigma_{d}$-{\it rotational and has} {\it rotation number} $p/q$. {\it The orbits are interlaced, meaning, given any} $n$ {\it consecutive elements of} $A$ {\it no two elements come from the same orbit}.
\end{proposition}

\begin{proof}
Use the correspondence $t_{i}\leftrightarrow i$ for $t_{i}\in A$ and $i\in \mathbb{Z}/q\mathbb{Z}$. Then the orbit of $t_{i}$ corresponds to the coset $H+i$ of the subgroup $H$ generated by $p$. Working over the integers, the least $m$ such that $mp\equiv 0$ (mod $q$) is $m=q/n$. Thus $|H|=q/n$ and the index of $H$ in $\mathbb{Z}/q\mathbb{Z}$ is $n$. Equivalently, $A$ is the union of $n$ orbits of size $m$.

For every $a\in\mathbb{Z}/q\mathbb{Z}$, the consecutive elements $a,a+1,\dots,a+n-1\in\mathbb{Z}/q\mathbb{Z}$ come from the consecutive cosets $H,H+1,\dots,H+n-1\in (\mathbb{Z}/q\mathbb{Z})/H$ (perhaps not in the same order) which are pairwise disjoint; this proves the interlacing property. Relabel the elements of $A$ by the rule $t_j^{(k)}=t_{jn+k-1}$, where $1\leq k\leq n$, so that
$$A=\{t_0^{(1)}<t_0^{(2)}<\dots<t_0^{(n)}<t_1^{(1)}<t_1^{(2)}<\dots<t_1^{(n)}<\cdots\cdots<t_{m-1}^{(n)}\}.$$
Applying the interlacing property to the collections $t_0^{(1)},\dots,t_0^{(n)}$ and $t_0^{(2)},\dots,t_1^{(1)}$ we see that $t_0^{(1)}$ and $t_1^{(1)}$ come from the same orbit. More generally, $t_{i}^{(k)}$ and $t_{j}^{(k)}$ belong in the same orbit. Picking out an orbit $\mathcal{O}\subset A$, we have
$\mathcal{O}=\{t_0^{(k)}<t_1^{(k)}<\dots<t_{m-1}^{(k)}\}$
and setting $m'=p/n$,
$$\sigma_d(t_j^{(k)})=\sigma_d(t_{jn+k-1})=t_{jn+k-1+p}=t_{(j+m')n+k-1}=t_{j+m'}^{(k)}.$$
Therefore $\mathcal{O}$ is $\sigma_d$-rotational with rotation number $m'/m=(p/n)/(q/n)=p/q$.
\end{proof} 
\end{section}

\begin{section}{Supporting lemmas}
\begin{definition}\label{def}
Let the $q$-tuple $(a_{0},a_{1},a_{2}, \dots,a_{q-1})$, $a_i\in\{0,1,\dots,d-1\}$, denote the element in $\mathbb{T}$ with base $d$ representation $0.\overline{a_{0}a_{1}a_{2}\ldots a_{q-1}}.$ To make calculations easier, we view the indices as elements of $\mathbb{Z}/q\mathbb{Z}$.
\end{definition}

Note that the lexicographical order on $q$-tuples agrees with their order as points in $\mathbb{T}$. The $q$-tuple representation has other nice properties. Notice that the map $\sigma_d$ acts on the set of $q$-tuples by $\sigma_d(a_0,a_1,a_2,\dots)=(a_1,a_2,a_3,\dots)$ which leads to the equation 
\begin{equation}\label{shift}
\sigma_d^k(a_0,a_1,\dots,a_{q-1})=(a_k,a_{k+1},\dots,a_{k+q-1}), \text{ $k\in\mathbb{Z}/q\mathbb{Z}$}.
\end{equation}
The notation $\sigma_d^k$ means $\sigma_d^{k'}$ for any representative $k'\in\mathbb{Z}$ of $k$. Suppose $A$ is a $\sigma_{d}$-rotational set with rotation number $p/q$ in lowest terms. Then $A$ is the union of orbits of size $q$ so that every point in $A$ has a $q$-tuple representation. By assumption gcd$(p,q)=1$, hence $p$ has a multiplicative inverse in $\mathbb{Z}/q\mathbb{Z}$ which we denote $p^*$. 

\begin{lemma}\label{iff}
Let $\mathcal{O}=\{t_{0}<t_{1}<t_{2}<\dots< t_{q-1}\}$ be an orbit in $\mathbb{T}.$ {\it Then} $\mathcal{O}$ {\it is} $\sigma_{d}$-{\it rotational with rotation number} $p/q$ in lowest terms {\it if and only if} $\sigma_{d}^{p^*}(t_{i})=t_{i+1}$ {\it for every} $t_{i}\in \mathcal{O}.$
\end{lemma}
\begin{proof}
Suppose $\mathcal{O}$ is $\sigma_{d}$-rotational with rotation number $p/q$ in lowest terms. Then for every $t_{i}\in \mathcal{O}$, we have $\sigma_{d}(t_{i})=t_{i+p}$ which implies $\sigma_{d}^{p^*}(t_{i})=t_{i+pp^*}=t_{i+1}$. Conversely, if $\sigma_{d}^{p^*}(t_{i})=t_{i+1}$ for every $t_{i}\in \mathcal{O}$ then $\sigma_{d}(t_{i})=(\sigma_{d}^{p^*})^{p}(t_{i})=t_{i+p}.$ The existence of $p^*$ implies that gcd$(p,q)=1$ so $p/q$ is in lowest terms.
\end{proof}

Define disjoint open intervals $I_{j}= ( \frac{j}{d},\frac{j+1}{d})\subset\mathbb{T}$ for $j=0,1,\dots,d-2$ and $I_{d-1} = (\frac{d-1}{d},\infty)\subset\mathbb{T}$. Notice that the boundary points of these intervals are precisely the points in the $\sigma_d$-preimage of $0$. Suppose the orbit of $t=(a_0,\dots,a_{q-1})$ is rotational. Since we exclude fixed points from the definition of rotational sets, $t\neq 0$. Moreover, $t\notin \sigma_d^{-1}(0)=\mathbb{T}\setminus(\cup_j I_j)$ or equivalently $t\in I_j$ for some $j$. It follows that $t\in I_j$ if and only if $j=a_0$.
\begin{lemma}\label{d}
Suppose $s=(a_0,a_1,\dots,a_{q-1})$ and $t=(b_0,b_1,\dots,b_{q-1})$ are distinct points of $\mathbb{T}$ such that their orbits are rotational sets. Then $a_0<b_0$ if and only if there exists $s<r<t$ such that $\sigma_d(r)=0$.
\end{lemma}
\begin{proof}

Suppose $a_0<b_0$. Since $s\in I_{a_0}$ and $t\in I_{b_0}$ there is a point in $r\in\sigma_d^{-1}(0)$ such that $s<r<t$. Conversely, suppose $r\in\sigma_d^{-1}(0)$ such that $s<r<t$. Since $s\in I_{a_0}$ and $t\in I_{b_0}$ we must have $I_{a_0}\neq I_{b_0}$, hence $a_0<b_0$.
\end{proof}

\begin{lemma}\label{lex}
Let $s=(a_0,a_1,\dots,a_{q-1})$ and $t=(b_0,b_1,\dots,b_{q-1})$. Fix an index $m\neq 0$. If $a_i\leq b_i$ for $i\neq m$ and $a_{m-1}<b_{m-1}$ then $s<t$.
\end{lemma}
\begin{proof}
Since $m\neq 0$, $m-1<m$. Now we use the lexicographical ordering and compare digits, noticing that $a_0\leq b_0, a_1\leq b_1,\dots,a_{m-2}\leq b_{m-2},a_{m-1}<b_{m-1}$. Therefore $s<t$.
\end{proof}

\begin{definition}
A $r$-{\it sequence} with \textit{length} $k$ is a nondecreasing sequence $a_{1}, a_{2}, a_{3},\dots,a_{k}$ with each $a_{i}$ chosen from $\{0,1,2,\dots,r-1\}.$ A simple counting argument shows that there are  ${{r-1+k} \choose {r-1}}$ $r$-sequences with length $k$.
\end{definition}

\begin{lemma}\label{bij}
There is a bijection between the set $\mathscr{A}$ of $r$-sequences with length $k$ such that $a_{i}<a_{i+1}$ for some fixed $1\leq i\leq r-1$ and the set $\mathscr{B}$ of $(r-1)$-sequences with length $k.$
\end{lemma}

\begin{proof}
Given a sequence $a_1,a_2,\dots,a_k$ in $\mathscr{A}$, subtract 1 from the terms $a_{i+1},a_{i+2},\dots,a_{k}$ to get a sequence in $\mathscr{B}$. Given a sequence $b_1,b_2,\dots,b_k$ in $\mathscr{B}$, add 1 to the terms $b_{i+1},b_{i+2},\dots,b_{k}$ to get a sequence in $\mathscr{A}$.
\end{proof}
\end{section}

\begin{section}{Main results}

\begin{theorem}\label{thm} For each $1\leq i\leq n$ let $\mathcal{O}_i$ be the orbit of $t_i\in\mathbb{T}$ with size $|\mathcal{O}_i|=q$, indexed so that $t_1<t_2<\dots<t_n$. We furthermore assume that the orbits $\mathcal{O}_i$ are distinct. The elements are written in $q$-tuple notation as

\begin{equation}\label{thmmat}
\begin{gathered}
t_{1} =(a_{0}^{(1)}, a_{1}^{(1)}, a_{2}^{(1)},\dots,a_{q-1}^{(1)}),\\
t_{2} =(a_{0}^{(2)}, a_{1}^{(2)}, a_{2}^{(2)},\dots,a_{q-1}^{(2)}),\\
t_{3} =(a_{0}^{(3)},a_{1}^{(3)}, a_{2}^{(3)},\dots,a_{q-1}^{(3)}),\\
\vdots \text{\qquad\qquad\qquad\qquad}\\
t_{n} =(a_{0}^{(n)}, a_{1}^{(n)}, a_{2}^{(n)}, \dots, a_{q-1}^{(n)}).
\end{gathered}
\end{equation}
Then each $t_i$ is the least element of $\mathcal{O}_i$ and the union of orbits $A=\cup_i \mathcal{O}_i$ is $\sigma_d$-rotational with rotation number $p/q$ in lowest terms if and only if
\begin{equation}\label{thmseq}
a_{0}^{(1)},a_{0}^{(2)},a_{0}^{(3)},\dots,a_{0}^{(n)}, a_{p^*}^{(1)}, a_{p^*}^{(2)},a_{p^*}^{(3)},\dots,a_{p^*}^{(n)},a_{2p^*}^{(1)}, a_{2p^*}^{(2)},a_{2p^*}^{(3)},\dots,a_{2p^*}^{(n)},\dots\dots, a_{(q-1)p^*}^{(n)}
\end{equation}
{\it is a} $d$-{\it sequence with} $a_{-(p+1)p^*}^{(n)}<a_{-pp^*}^{(1)}$.
\end{theorem}
We get the following corollary by restricting to the case $n=1$.
\begin{corollary}\label{cor}
Let $\mathcal{O}$ be the orbit of $t=(a_0,a_1,a_2,\dots,a_{q-1})$. Then $t$ is the least element of the $\sigma_d$-rotational orbit $\mathcal{O}$ with rotation number $p/q$ in lowest terms if and only if
\begin{equation}\label{ss}
    a_0\leq a_{p^*}\leq a_{2p^*}\leq \dots \leq a_{-(p+1)p^*} < a_{-pp^*}\leq\dots\leq a_{(q-1)p^*}.
\end{equation}
\end{corollary}
\begin{proof}
Suppose that $t_i$ is the least element of $\mathcal{O}_i$ and that $A$ is $\sigma_{d}$-rotational with rotation number $p/q$ in lowest terms. By Lemma \ref{iff} and by the interlacing property of Proposition \ref{inter} we can list the elements of $A$ in increasing order as
\begin{equation}\label{els}
t_1,t_2,\dots,t_n,\sigma_d^{p^*}(t_1),\sigma_d^{p^*}(t_2),\dots,\sigma_d^{p^*}(t_n),\sigma_d^{2p^*}(t_1),\sigma_d^{2p^*}(t_2),\dots\dots,\sigma_d^{(q-1)p^*}(t_n).
\end{equation}
From equation (\ref{shift}), we see that picking the ``leading digit'' (the first number that appears in the $q$-tuple) of each element above recovers (\ref{thmseq}). Since $q$-tuples respect lexicographical ordering, (\ref{thmseq}) is a $d$-sequence. Now notice that
$$\sigma_{d}(a_{(q-1)p^*-1}^{(n)},a_{(q-1)p^*}^{(n)},\dots)=(a_{(q-1)p^*}^{(n)},\dots) \quad\text{ and }\quad \sigma_{d}(a_{q-1}^{(1)},a_{0}^{(1)},\dots)=(a_{0}^{(1)},\dots).$$
\noindent Let $\gamma:[0,1]\to \mathbb{T}$ be a monotonically increasing path onto $J=[(a_{(q-1)p^*-1}^{(n)},\dots),(a_{q-1}^{(1)},\dots)]$. Since $(a_{(q-1)p^*}^{(n)},\dots)>t_1=(a_{0}^{(1)},\dots)$ the image of the path $\sigma_d\circ\gamma$ contains $0$. Then $J$ contains an element of the $\sigma_d$-preimage of 0, so by Lemma \ref{d} we obtain
$$
a_{-(p+1)p^*}^{(n)}=a_{(q-1)p^*-1}^{(n)}<a_{q-1}^{(1)}=a_{-pp^*}^{(1)}.
$$

Conversely, suppose (\ref{thmseq}) is a $d$-sequence with $a_{-(p+1)p^*}^{(n)}<a_{-pp^*}^{(1)}$. Observe that going down a column of (\ref{thmmat}) corresponds to going across a segment of (\ref{thmseq}). Thus $a_{i}^{(j)}\leq a_{i}^{(j+1)}$ for any $i\in \mathbb{Z}/q\mathbb{Z}$ and $1\leq j\leq n-1$. Then we have $t_{1}<t_{2}<\dots<t_{n}$ since the $t_{j}$ are distinct. More generally, $\sigma_{d}^{kp^*}(t_{1})<\sigma_{d}^{kp^*}(t_{2})<\dots<\sigma_{d}^{kp^*}(t_{n})$ for any integer $k$. We wish to additionally show that $\sigma_{d}^{kp^*}(t_{n})<\sigma_{d}^{(k+1)p^*}(t_{1})$ for $0\leq k\leq q-2$ so that (\ref{els})
is increasing. By Lemma \ref{iff} this would imply that $A$ is $\sigma_{d}$-rotational with rotation number $p/q$ in lowest terms, and that $t_i$ is the least element of $\mathcal{O}_i$. To this end, notice that for $0\leq k\leq q-2$
\begin{align}\label{pair}
\begin{split}
\sigma_{d}^{kp^*}(t_{n}) &=(a_{kp^*}^{(n)},\ a_{kp^*+1}^{(n)},\ a_{kp^*+2}^{(n)},\ \dots,\ a_{kp^*+q-1}^{(n)}),\\
\sigma_{d}^{(k+1)p^*}(t_{(1)}) &=(a_{(k+1)p^*}^{(1)},\  a_{(k+1)p^*+1}^{(1)},\  a_{(k+1)p^*+2}^{(1)},\ \dots,\ a_{(k+1)p^*+q-1}^{(1)}).
\end{split}
\end{align}
From (\ref{thmseq}) we have
$$
a_{kp^*+i}^{(n)}=a_{(k+ip)p^*}^{(n)}\leq a_{(k+ip+1)p^*}^{(1)}=a_{(k+1)p^*+i}^{(1)}
$$
unless $k+ip=q-1$. But $k+ip=q-1$ implies that
$$a_{kp^*+i-1}^{(n)}=a_{(k+ip-p)p^*}^{(n)}=a_{-(p+1)p^*}^{(n)}<a_{-pp^*}^{(1)}=a_{(k+ip+1-p)p^*}^{(1)}=a_{[k+1+(i-1)p]p^*}^{(1)}=a_{(k+1)p^*+i-1}^{(1)}.$$
Applying Lemma \ref{lex} to the pair (\ref{pair}) we see that $\sigma_{d}^{kp^*}(t_{n})<\sigma_{d}^{(k+1)p^*}(t_{0})$ except when $k+ip=q-1$ and $i=0$, i.e. when $k=q-1.$
\end{proof}
Given a $\sigma_d$-rotational orbit $\mathcal{O}$ with rotation number $p/q$ in lowest terms, we can choose its least element $t=(a_0,a_1,\dots,a_{q-1})$ which is associated with 
the $d$-sequence $(\ref{ss})$. Conversely, given a $d$-sequence of the form $(\ref{ss})$ we have $t=(a_0,a_1,\dots,a_{q-1})$ being the least element of its $\sigma_d$-rotational orbit. We can make this simpler since $(\ref{ss})$ corresponds to the $(d-1)$-sequence
\begin{equation}\label{rep}
a_{0},a_{p^*},\dots,a_{-(p+1)p^*},a_{-pp^*}-1,\dots,a_{(q-2)p^*}-1, a_{(q-1)p^*}-1
\end{equation}
via the bijection from Lemma \ref{bij}.
\begin{definition}
Fix a rotation number $p/q$ given in lowest terms. If $t=(a_0,a_1,\dots,a_{q-1})$ is the least element of its $\sigma_d$-rotational orbit $\mathcal{O}$ with rotation number $p/q$, we call $(\ref{rep})$ the {\it representative sequence} of $\mathcal{O}$. By Corollary \ref{cor} and Lemma \ref{bij} the set of $\sigma_d$-rotational orbits with rotation number $p/q$ is in bijection with the set of representative sequences, which is equal to the set of $(d-1)$-sequences of length $q$. We say that a collection of representative sequences $s_{i}=b_{0}^{(i)},b_{1}^{(i)},\dots,b_{q-1}^{(i)}$ for $1\leq i\leq n$ \textit{can be interlaced} if
$$
b_{0}^{(1)},b_{0}^{(2)},b_{0}^{(3)},\dots,b_{0}^{(n)},b_{1}^{(1)},b_{1}^{(2)},b_{1}^{(3)},\dots,b_{1}^{(n)},b_{2}^{(1)},b_{2}^{(2)},b_{2}^{(3)},\dots,b_{2}^{(n)},\dots\dots,b_{q-1}^{(n)}
$$
is a $(d-1)$-sequence, up to relabeling.
\end{definition}
\begin{corollary}
    For $1\leq i\leq n$, let $\mathcal{O}_i$ be a $\sigma_d$-rotational orbit with rotation number $p/q$. Then $A=\cup_i \mathcal{O}_i$ is $\sigma_d$-rotational with rotation number $p/q$ if and only if the associated collection of representative sequences $\{s_i\}_{i=1}^n$ can be interlaced.
\end{corollary}
\begin{proof}
    Follows from Theorem \ref{thm}.
\end{proof}
McMullen in [5] gives an explicit algorithm for computing rotational orbits for a fixed rotation number $p/q$. Example \ref{alg} describes this same algorithm. In fact, Theorem \ref{thm} gives us a something stronger: we can compute all rotational \textit{sets} by checking if representative sequences can be interlaced then taking unions of the corresponding rotational orbits.

\begin{exmp}\label{alg}
Let us construct a $\sigma_{4}$-rotational orbit with rotation number $\frac{2}{5}$. The multiplicative inverse of 2 in $\mathbb{Z}/ 5\mathbb{Z}$ is 3. Choose a representative sequence, say the 3-sequence $0,1,1,1,2.$ Add 1 to the last 2 terms to get $0,1,1,2,3$. Substituting these numbers for $a_0,a_3,a_6=a_1,a_9=a_4,a_{12}=a_2$ in $(a_0,a_1,a_2,a_3,a_4)$ we obtain the 5-tuple $(0,1,3,1,2)$ which corresponds to the rational number
$\frac{118}{1028}$. This gives the $\sigma_4$-rotational orbit $\{\frac{118}{1028},\frac{391}{1028},\frac{472}{1028},\frac{541}{1028},\frac{865}{1028}\}$ with rotation number $\frac{2}{5}$.
\end{exmp}

\begin{exmp}\label{alg2}
The undirected graph in Figure 1 shows all $\sigma_{4}$-rotational sets with rotation number $p/4$ for any $p\in \mathbb{Z}/ 4\mathbb{Z}$ with gcd$(p,4)=1$. The vertices represent $\sigma_{4}$-rotational orbits, labeled by their representative sequences. Two vertices are connected by an edge if and only if their representative sequences can be interlaced. It follows that any set of $n$ mutually adjacent vertices, or $n$-{\it clique}, gives sequences which can be interlaced. Therefore each $n$-clique in the graph corresponds to a $\sigma_{4}$-rotational set which is a union of $n$ orbits. There are 15 1-cliques, 30 2-cliques, and 16 3-cliques in the graph. Altogether there are 61 $\sigma_{4}$-rotational sets with rotation number $p/4.$
\end{exmp}

\begin{figure}[t]
\begin{center}
\begin{tikzpicture}
\coordinate (offset) at (-1,0.5);

\coordinate (1) at (0*2,5);

\coordinate (21) at (-0.5*2,4);
\coordinate (22) at (0.5*2,4);

\coordinate (31) at (-1*2,3);
\coordinate (32) at (0,3);
\coordinate (33) at (1*2,3);

\coordinate (41) at (-1.5*2,2);
\coordinate (42) at (-0.5*2,2);
\coordinate (43) at (0.5*2,2);
\coordinate (44) at (1.5*2,2);

\coordinate (51) at (-2*2,1);
\coordinate (52) at (-1*2,1);
\coordinate (53) at (0,1);
\coordinate (54) at (1*2,1);
\coordinate (55) at (2*2,1);

\draw (1) -- (0.5*2 -0.55,4 +0.55);
\draw (21) -- (0 -0.55, 3 +0.55);
\draw (22) -- (1*2 -0.55, 3 +0.55);
\draw (31) -- (-0.5*2 -0.55, 2 +0.55);
\draw (32) -- (0.5*2 -0.55, 2 +0.55);
\draw (33) -- (1.5*2 -0.55, 2 +0.55);
\draw (41) -- (-1*2 -0.55, 1 +0.55);
\draw (42) -- (0 -0.55, 1 +0.5);
\draw (43) -- (1*2 -0.55, 1 +0.55);
\draw (44) -- (2*2 -0.55, 1 +0.55);

\node[above left] at (1)  {0000};

\node[above left] at (21) {0001};
\node[above left] at (22) {0002};

\node[above left] at (31) {0011};
\node[above left] at (32) {0012};
\node[above left] at (33) {0022};

\node[above left] at (41) {0111};
\node[above left] at (42) {0112};
\node[above left] at (43) {0122};
\node[above left] at (44) {0222};

\node[above left] at (51) {1111};
\node[above left] at (52) {1112};
\node[above left] at (53) {1122};
\node[above left] at (54) {1222};
\node[above left] at (55) {2222};

\filldraw (1) circle (2.5pt);

\filldraw (21) circle (2.5pt);
\filldraw (22) circle (2.5pt);

\filldraw (31) circle (2.5pt);
\filldraw (32) circle (2.5pt);
\filldraw (33) circle (2.5pt);

\filldraw (41) circle (2.5pt);
\filldraw (42) circle (2.5pt);
\filldraw (43) circle (2.5pt);
\filldraw (44) circle (2.5pt);

\filldraw (51) circle (2.5pt);
\filldraw (52) circle (2.5pt);
\filldraw (53) circle (2.5pt);
\filldraw (54) circle (2.5pt);
\filldraw (55) circle (2.5pt);

\draw (21) -- (22);
\draw (31) -- (33);
\draw (41) -- (44);
\draw (51) -- (55);

\draw (1) -- (51);
\draw (22) -- (52);
\draw (33) -- (53);
\draw (44) -- (54);

\end{tikzpicture}
\end{center}
\caption{ Graph representing $\sigma_{4}$-rotational sets with rotation number $\frac{1}{4}$ or $\frac{3}{4}$}
\end{figure}
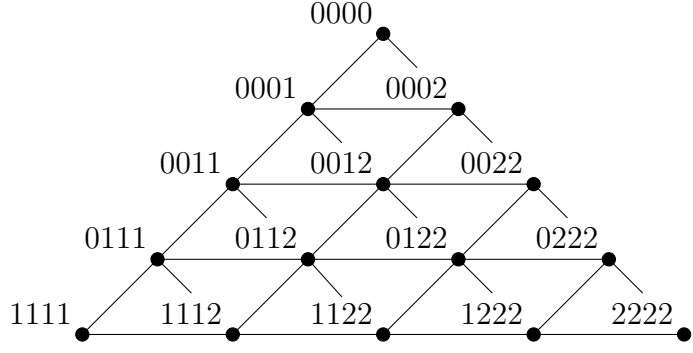

The graph in Figure 1 has a nice structure: it can be embedded in $\mathbb{R}^2$ to make a simplicial complex which is a subdivision of the triangle or $2$-simplex. It turns out to be true that the set of $\sigma_d$-rotational sets with fixed rotation number can be naturally embedded as the vertices of a simplicial complex which subdivides the $(d-2)$-simplex. This phenomenon was noticed in [5]. A proof is provided in [7]. The simplicial subdivision obtained in this way coincides with the \textit{edgewise subdivision} from [4].

\begin{corollary}\label{cor}
{\it There are} ${{d-2+q}\choose{d-2}}$ $\sigma_{d}$-{\it rotational orbits with rotation number} $p/q$ in lowest terms.
\end{corollary}

\begin{proof}
There are ${{d-2+q}\choose{d-2}}$ representative sequences each corresponding to a unique orbit.
\end{proof}

\begin{corollary}
{\it A} $\sigma_{d}$-{\it rotational set contains at most} $d-1$ {\it orbits}.
\end{corollary}

\begin{proof}
Suppose that the union of $n$ orbits with rotation number $p/q$ is a $\sigma_{d^{-}}$rotational set. Then there exists a  $(d-1)$-sequence with length $nq$ which can be unlaced to give $n$ distinct $(d-1)$-sequences with length $q$. Let $s'=b_{0},b_{1},b_{2},\dots,b_{nq-1}$ be this $(d-1)$-sequence and let

$$s_{i}=b_{i},b_{i+n},b_{i+2n},\dots,b_{i+(q-1)n}\qquad\text{ for } 0\leq i\leq n-1.$$

\noindent Since the $s_{i}$ are distinct, adjacent sequences in $s_{0},s_{1},\dots,s_{n-1}$ must differ by at least one term. Then going across $s'$ we should find at least $n-1$ pairs of nonequal adjacent terms. But since $s'$ is a $(d-1)$-sequence (there are only $d-1$ possibilities for the terms), there can be at most $d-2$ such pairs. This shows that $n-1 \leq d-2$ or $n\leq d-1$.
\end{proof}

\begin{theorem}\label{main}
{\it Let} $N_{k}$ {\it be the number of} $\sigma_{d}$-{\it rotational sets with rotation number} $p/q$ (in lowest terms) {\it that contain precisely} $k$ {\it orbits. Each} $N_{k}$ {\it is given recursively by}

\begin{align}\label{form}
\begin{split}
N_{1} &={{d-2+q}\choose{d-2}},\text{ }
{\it and}\\
N_{k} &={{d-2+kq}\choose{d-2}} - \sum_{j=1}^{k-1} {{k-1}\choose{j-1}} N_{j},\text{ for } 2\leq k\leq d-1.
\end{split}
\end{align}
\end{theorem}

\begin{proof}
We know $N_{1}={{d-2+q}\choose{d-2}}$ from Corollary \ref{cor}. Let $k\in \mathbb{Z}$ such that $2\leq k\leq d-1$. Each rotational set counted by $N_{k}$ can be identified by a $(d-1)$-sequence with length $kq$ obtained by interlacing $k$ distinct $(d-1)$-sequences, each with length $q$. By Lemma 2.4 the number of $(d-1)$-sequences with length $kq$ is ${{d-2+kq}\choose{d-2}}$. To find $N_{k}$, we want to subtract from ${{d-2+kq}\choose{d-2}}$ the number of sequences obtained by interlacing nondistinct representative sequences. Let $s'=b_{0},b_{1},b_{2},\dots,b_{kq-1}$ be a  $(d-1)$-sequence with length $kq$ and let

$$s_{i}=b_{i},b_{i+k},b_{i+2k},\dots,b_{i+(q-1)k} \qquad\text{ for } 0\leq i\leq k-1.$$

\noindent Now, because $s'$ is nondecreasing, identical sequences in $s_{0},s_{1},\dots,s_{k-1}$ group together with consecutive indices. Suppose there are precisely $j$ distinct sequences in $s_{0}, s_{1},\dots,s_{k-1}$, where $1< j\leq k-1$. There are $N_{j}$ possibilities for these distinct sequences. Furthermore, there are ${{k-1}\choose{j-1}}$ possibilities for how they are grouped, given by a choice of $j-1$ dividers in between elements of the list $s_{0},s_{1},\dots,s_{k-1}$. Thus, we subtract ${{k-1}\choose{j-1}} N_{j}$ from ${{d-2+kq}\choose{d-2}}$ for each $1\leq j\leq k-1$.
\end{proof}
The application of a combinatorial inversion formula (Proposition \ref{app}) gives us the following nonrecursive version of (\ref{form}).
\begin{equation*}
N_k = \sum_{j=1}^k (-1)^{k+j} {{k-1}\choose{j-1}}{{d-2+jq}\choose{d-2}}
\end{equation*}

\begin{proposition}\label{app}
Let $b=[b_1,\dots,b_m]^t\in\mathbb{R}^m$ be a column vector with real entries. The formulas
\begin{equation}\label{f1}
n_{1} =b_1,\quad
n_{k} =b_k - \sum_{j=1}^{k-1} {{k-1}\choose{j-1}} n_{j}\text{ for } 2\leq k\leq m
\end{equation}
and
\begin{equation}\label{f2}
n_k = \sum_{j=1}^k (-1)^{k+j} {{k-1}\choose{j-1}}b_j
\end{equation}
define the same vector $n=[n_1,\dots,n_m]^t\in\mathbb{R}^m$.
\end{proposition}
\begin{proof}
Let $L=(a_{ij})$ and $L'=(a'_{ij})$ be the $m\times m$ lower triangular matrices defined by $a_{ij}={{i-1}\choose{j-1}}$ and $a'_{ij}=(-1)^{i+j}{{i-1}\choose{j-1}}$. For any $x\in\mathbb{R}$ the binomial theorem gives
\begin{align*}
L[1,x,\dots,x^{m-1}]^t &=[1,x+1,\dots,(x+1)^{m-1}]^t, \\
L'[1,x,\dots,x^{m-1}]^t &=[1,x-1,\dots,(x-1)^{m-1}]^t.
\end{align*}
It follows that $L'=L^{-1}$ (by properties of the Vandermonde matrix). Formula (\ref{f1}) produces the solution $n$ of $Ln=b$ using forwards substitution, whereas (\ref{f2}) can be written as $n=L^{-1}b$.
\end{proof}

\end{section}
\begin{section}{Acknowledgements}

I would like to thank John Mayer and Debra Gladden for introducing me to rotation sets. I also thank Debra Gladden for her support and encouragement throughout my undergraduate years. I thank Saeed Zakeri for his helpful comments on a draft of this paper.
\end{section}
\section*{References}

\:\:\:\:\, [1] Alexander Blokh, James Malaugh, John Mayer, Lex Oversteegen, and Daniel Parris. Rotational subsets of the circle under $z^{d}$. {\it Topology and its Applications}, 153(1):1540-1570, 2006.

\medskip

[2] Lisa Goldberg. Fixed points of polynomial maps. I. Rotation subsets of the circles. {\it Annales} {\it scientifiques de l}'{\it École Normale Supérieure}, Ser. 4, 25(6):679-685, 1992.

\medskip

[3] Lisa Goldberg and John Milnor. Fixed points of polynomial maps. Part II. Fixed point portraits. {\it Annales scientifiques de l}'{\it cole Normale Suprieure}, 26(1):51-98, 1993.

\medskip

[4] Edelsbrunner, H., Grayson, D.R. Edgewise Subdivision of a Simplex. \textit{Discrete Comput Geom} \textbf{24}, 707–719 (2000).

\medskip

[5] Curtis McMullen. Dynamics on the unit disk: Short geodesics and simple cycles. \textit{Commentarii Mathematici Helvetici}, 85:723-749, 2010.

\medskip

[6] Carsten L. Petersen and Saeed Zakeri. On combinatorial types of periodic orbits of the map {\it x} $\mapsto kx$ (mod $\mathbb{Z}$). Part I: Realization, arXiv:1712.04506, 2017.

\medskip

[7] Saeed Zakeri, \textit{Rotation sets and complex dynamics}, Lecture Notes in Mathematics, vol. 2214, Springer, Cham, 2018.

\end{document}